\documentclass{amsart}
\usepackage{amsmath}
\usepackage{latexsym}
\usepackage{float}
\usepackage{amssymb}
\usepackage[all]{xy}
\usepackage{epsfig}
\usepackage{color}
\newtheorem{Thm}[equation]{Theorem}
\newtheorem{Lem}[equation]{Lemma}
\newtheorem{Cor}[equation]{Corollary}
\newtheorem{Prop}[equation]{Proposition}

\theoremstyle{remark}

\theoremstyle{definition}

\numberwithin{equation}{section}

\newcommand{\R}{\text{\bf R}}           
\newcommand{\C}{\text{\bf C}}           
\newcommand{\N}{\text{\bf N}}           
\newcommand{\Z}{\text{\bf Z}}
\newcommand{\E}{\text{\bf E}}
\newcommand{\T}{\text{\bf T}}

\newcommand{\ad}{\text{ad}}

\newcommand{\Tr}{\text{Tr}}

\newcommand{\Ind}{\text{Ind}}
\renewcommand{\Re}{\text{Re}}

\newcommand{\fa}{{\mathfrak a}}             
\newcommand{\fb}{{\mathfrak b}}
\newcommand{\fg}{{\mathfrak g}}

\newcommand{\fk}{{\mathfrak k}}
\newcommand{\fl}{{\mathfrak l}}
\newcommand{\fm}{{\mathfrak m}}
\newcommand{\fn}{{\mathfrak n}}

\newcommand{\fs}{{\mathfrak s}}

\newcommand{\be}{\begin{equation}}
\newcommand{\beu}{\begin{equation*}}

\renewcommand{\bar}[1]{\overline{#1}}

\begin{document}
 
\baselineskip=16pt

\title{A functional equation and degenerate principal series}
\author{Juhyung Lee}
\address{Oklahoma State University \\ Mathematics Department \\ Stillwater, Oklahoma 74078}
\email{juhylee@math.okstate.edu}

\begin{abstract}
A functional equation between the zeta distributions can be obtained from the theory of prehomogeneous vector spaces. We show that the functional equation can be extended from the Schwartz space to certain degenerate principal series. 
\end{abstract}

\maketitle
\begin{center}\today \end{center}

\section{Introduction}
The fundamental theorem of prehomogeneous vector spaces gives a functional equation concerning the Fourier transformation of
a complex power of the associated relative invariants as a distribution. 
For certain reductive Lie groups that we consider, the corresponding prehomogeneous vector space gives the functional equation of the form (\ref{eq:F_E_Schwartz}) below.  
The purpose of this paper is to extend the functional equation between the tempered distributions to functions in certain degenerate principal series representations.
The choices for the groups $G$ on which we define degenerate principal series representations are 
given on Table $1$.
Such groups $G$ arise from simple non-Euclidean Jordan algebras $\fn$ in the sense that
each $\fn$ occurs as the abelian nilradical of a maximal parabolic subalgebra of a reductive Lie algebra $\fg$, where 
$\fg=\textup{Lie}(G)$ for some reductive Lie group $G$ in the table.

Each group $G$ has a parabolic subgroup $P=LN$ ($N$ is abelian) for which the adjoint action of $L$ on $\fn=\textup{Lie}(N)$ has a finite number of orbits.
In particular, there is an open orbit which is in fact dense and has complement defined by a polynomial equation $\nabla^2(Y)=0$.
Therefore, the triple $(L, \textup{Ad}, \fn)$ is a prehomogeneous vector space and its contragradient action of $L$ on the Lie algebra $\bar{\fn}=\textup{Lie}(\bar{N})$ of the opposite nilradical $\bar{N}$ has an open orbit which is again dense and has complement defined by a polynomial equation $\bar{\nabla}^2(X)=0$.
For example, for $G=GL(2n,\R)$, $\nabla=\bar{\nabla}=|\det|$.
As tempered distributions, the zeta distributions are defined by the integrals
\begin{align*}
\bar{\Z}(f,t)&=\int_{\bar{\fn}}f(X)\bar{\nabla}(X)^tdX, \textup{ for }f\in \mathcal{S}(\bar{\fn})\textup{ and}\\
\Z(h,t)&=\int_{\fn}h(Y)\nabla(Y)^tdY, \textup{ for }h\in \mathcal{S}(\fn).
\end{align*}
Here $\mathcal{S}(\bar{\fn})$ (resp. $\mathcal{S}(\fn)$) denotes the Schwartz space on $\bar{\fn}$ (resp. $\fn$).
The integers $m, n, d$ and $e$ are listed on Table $1$ for each group.
It is well-known that $\bar{\Z}(f,t)$ and $\Z(h,t)$ converge absolutely for $\textup{Re}(t)>-(e+1)$ and both expressions are complex analytic functions of $t$.
Moreover, these analytic functions in $t$ extend meromorphically to the whole complex plane and satisfy the functional equation
 \begin{equation}\label{eq:F_E_Schwartz}
\frac{\pi^{\frac{nt}{2}}}{\Gamma_n(t)}\bar{\Z}(f,t-\frac{m}{n})
=\frac{\pi^{\frac{n}{2}(-t+\frac{m}{n})}}{\Gamma_n(-t+\frac{m}{n})}\Z(\widehat{f},-t), \textup{ for }f\in \mathcal{S}(\bar{\fn}).
\end{equation}
Here $\Gamma_n(t)=\prod_{j=0}^{n-1}\Gamma\left(\frac{t-jd}{2}\right)$ ($\Gamma$ is the gamma function on $\C$) and 
$\,\,\widehat{}\,\,$ denotes the Fourier transform.
The functional equation concerning the Fourier transform of a complex power of the relative invariants as a tempered distribution was established by Mikio Sato and is called the fundamental theorem of prehomogeneous vector spaces. 
See, for example, \cite[Theorem 4.17]{Kimura03}. 
For our choices of the groups $G$, the functional equation (\ref{eq:F_E_Schwartz}) is contained in \cite{Muller86}.

We define the family of degenerate principal series representations for $s\in \C$ and each can be realized as the space of certain functions on $\bar{\fn}$. We denote it by $I(s)$. Note that the Schwartz space $\mathcal{S}(\bar{\fn})$ is contained in $I(s)$ for all $s$.
In this paper we consider the integrals $\bar{\Z}(F_s,t)$ for $F_s\in I(s)$ and prove:

\begin{Thm}\label{thm:mero_non}
Let $F_s\in I(s)$. Then the family of integrals $\bar{\Z}(F_s,t)$ is complex analytic on $-(e+1)<\textup{Re}(t)<\textup{Re}(s)-d(n-1)$
and has a meromorphic continuation to all of $\C^2$.
\end{Thm}

\begin{Thm}\label{thm:extended_F_E}
Let $F_s\in I(s)$. Then the functional equation
\begin{equation}\label{eq:FT_non}
\frac{\pi^{\frac{nt}{2}}}{\Gamma_n(t)}\bar{\Z}(F_s,t-\frac{m}{n})
=\frac{\pi^{\frac{n}{2}(-t+\frac{m}{n})}}{\Gamma_n(-t+\frac{m}{n})}\Z(\widehat{F_s},-t)
\end{equation}
holds as meromorphic functions in $(s,t)\in \C^2$.
\end{Thm}

As for the organization of our paper, we set up some notations and give some properties about the groups that we consider together with an integral formula for `polar coordinates' in section $2$ in a form convenient for our purposes. 
We define principal series representations $I(s)$ and give formulas for the action by elements of the $n$ commuting copies of $\mathfrak{sl}(2,\R)$ in section $3$. Sections $4$ contains the functional equation between the zeta distributions and the convergence range of the integral $\bar{\Z}(F_s,t)$ for $F_s\in I(s)$. Section $5$ contains the proofs of the main results.
We apply the polar coordinates to the integral $\bar{\Z}(F_s,t)$ to reduce it to the integrals over a noncompact radial set and a compact set. Then the main part of the proof is to show the meromorphic continuation of the integral over the noncompact set. 
Our main technique is to apply a string of differential operators to extend the defining range of the integral $\bar{\Z}(F_s,t)$ in two variables $s$ and $t$. Such differential operators are obtained by the Lie algebra action in section $3$.
Section $6$ contains applications of my main results to the representation theory.

Our method is inspired by the treatment of $SL(2,\R)$ in \cite{Lee2012}.
For each $0<s<1$ the complementary series representation $I(s)$ of $SL(2,\R)$ is unitary 
because there is an invariant Hermitian form $\langle \cdot,\cdot \rangle$ 
induced from the `standard' intertwining map $A_s(F_s)(y)$, 
which can be expressed as the integral $\bar{\Z}\left(\tau_yF_s,s-1\right)$ for $F_s\in I(s)$, 
here $\tau_yF_s=F_s(\cdot+y)$, 
and it is seen to be positive definite for $0<s<1$ in the following way.
As convergent integrals for $F_s\in I(s)$, $0<s<1$, we have the following string of equalities:
\begin{align}
\langle F_s,F_s \rangle&=\int_{\R}F_s(y)\overline{\frac{\pi^{\frac{s}{2}}}{\Gamma(\frac{s}{2})}\Z(\tau_yF_s,s-1)}dy\nonumber\\
&=\int_{\R}F_s(y)\overline{\frac{\pi^{\frac{1-s}{2}}}{\Gamma(\frac{1-s}{2})}\Z(\widehat{\tau_yF_s},-s)}dy\label{f_e_n=1}\\
&=\frac{\pi^{\frac{1-s}{2}}}{\Gamma(\frac{1-s}{2})}\int_{\R}\int_{\R}F_s(y)e^{-2\pi xy}\overline{\widehat{F}_s}(x)|x|^{-s}dxdy\nonumber\\
&=\frac{\pi^{\frac{1-s}{2}}}{\Gamma(\frac{1-s}{2})}\int_{\R}|\widehat{F}_s|^2(x)|x|^{-s}dx.\label{formula_inner_product_I(s)_1}\nonumber
\end{align}
The equation (\ref{f_e_n=1}) is from the functional equation (\ref{eq:F_E_Schwartz}), which is also valid for functions in $I(s)$ in the case of $SL(2,\R)$.
Therefore, a continuation of $\bar{\Z}$ for functions in $I(s)$ and 
an associated functional equation are the first step 
to construct the positive definite of the invariant Hermition form.
The proof of Proposition \ref{prop:herm_mero} is done by a similar string of equalities for $G=GL(2n,\R)$.  
In \cite{BarchiniSepanskiZierau06}, for our choices of the groups the functional equation for Schwartz functions (\ref{eq:F_E_Schwartz}) is used to construct unitary representations for $s \in [0,e+1)\cup\{\frac{m}{n}-qd:q=0,1,\cdots,n-1\}$.
See also \cite{DvorskySahi99} and \cite{DvorskySahi03} for a different approach of the realizations. 

This work was inspired in part by \cite{Lee2012} and I thank to my advisor Dr. Zierau for helping me throughout the project.

\section{Preliminaries}

The choices for the groups $G$ which we work with are given in Table $1$ in Appendix $\bf A$.
These groups $G$ arise from simple non-Euclidean Jordan algebras.
Such a $G$ is characterized by the existence of a parabolic subgroup $P=LN$ (a Levi decomposition) such that
$P$ and its opposite parabolic $\bar{P}=L\bar{N}$ are $G$-conjugate, $N$ is abelian, 
and the symmetric space corresponding to $G$ is not of tube type.
$G$ has a Cartan involution $\theta$ so that $\theta$ sends $P$ to the opposite parabolic.
We let $K$ be the fixed point group of $\theta$, a maximal compact subgroup of $G$.
We write the real Lie algebras of various Lie groups by the corresponding fraktur letters. 
The Cartan involution determines a Cartan decomposition $\fg=\fk+\fs$.

Following \cite{KosSah93} there is a maximal abelian subalgebra $\fb$ of $\fl \cap \fs$ with the following properties.

($1$) There are commuting copies of $\mathfrak{sl}(2,\R)$, denoted by $\mathfrak{sl}(2,\R)_j$, in $\fg$ spanned by $\{F_j,H_j,E_j\}$ 
 a standard basis in the sense that
\begin{align*}
&\theta(E_j)=-F_j\quad \textup{and}\quad \theta(H_j)=-H_j\\
&[E_j,F_j]=H_j, [H_j,E_j]=2E_j \quad \textup{and}\quad [H_j,F_j]=-2F_j\,,
\end{align*}
with $E_j\in \fn$, $F_j\in \bar{\fn}$ and $\fb=\sum_{j=1}^n \R H_j$.
Therefore, we can view $\prod_{j=1}^n\mathfrak{sl}(2,\R)_j$ as a subalgebra of $\fg$ 
and will denote the element $aH_j+bE_j+cF_j$ by $\begin{pmatrix}a&b\\c&-a\end{pmatrix}_j$.
Also, we can find the correspond Lie group $\prod_{j=1}^n \textup{SL}(2,\R)_j$ in $G$ and define,
for $\begin{pmatrix}a&b\\c&d\end{pmatrix}\in \textup{SL}(2,\R)$,
\begin{equation*}
\begin{pmatrix}a&b\\c&d\end{pmatrix}_j=I\times \cdots \times \begin{pmatrix}a&b\\c&d\end{pmatrix} \times \cdots \times I.
\end{equation*}

($2$) For $\epsilon_k(\sum_{j=1}^n a_j H_j)\equiv a_k$, the $\fb$-roots in $\fg$, $\fl$, and $\fn$ are 
\begin{align*}
\Sigma(\fg,\fb)&=\{\pm(\epsilon_j-\epsilon_k):1\leq j<k\leq n\}\cup\{\pm(\epsilon_j+\epsilon_k):1\leq j,k\leq n\}   \\
\Sigma(\fl,\fb)&=\{\pm(\epsilon_j-\epsilon_k):1\leq j<k\leq n\}\quad\textup{and}   \\
\Sigma(\fn,\fb)&=\{\epsilon_j+\epsilon_k:1\leq j,k\leq n\}.
\end{align*}

For each $G$, the roots in $\fn$ have just two possibilities for the multiplicity. This define integers $d$ and $e$:
\begin{align*}
&\textup{each short root has multiplicity $2d$ and}\\
&\textup{each long root has multiplicity $e+1$}.
\end{align*}
(In the case of $SO(p,q)$, Case $4$ on Tables $1$ and $2$, $d$ is a half integer. When $n=1$, $d$ is zero.)
We set $m=\textup{dim}(\fn)$. Then $m=n(d(n-1)+(e+1))$.
Define
\begin{equation*}
\Sigma^+(\fg,\fb)=\{\epsilon_j-\epsilon_k:1\leq j<k\leq n\}\cup\Sigma(\fn,\fb).
\end{equation*}

Define a character $\Lambda_0$ on $\fb$ as $\Lambda_0\equiv \sum_{j=1}^n\epsilon_j$. 
Then $\Lambda_0$ extends to a character of $\fl$
and we write $e^{\Lambda_0}$ for the corresponding character of $L$. 

There is a diffeomorphism of $\bar{\fn}\times L\times \fn$ onto a dense open set in $G$ given by $(X,\ell,Y)\mapsto\bar{\fn}_X\ell\fn_Y$,
where $\bar{n}_X=\textup{exp}(X)$ and $n_Y=\textup{exp}(Y)$. 
Therefore, any $g\in \bar{N}LN$ has a unique decomposition as $g=\bar{N}(g)\ell(g)N(g)$.
Furthermore, $L=MA$, where $A=\textup{exp}(\fa)$, $\fa=\bigcap_{j<k}ker(\epsilon_j-\epsilon_k)$.
Since the $L$ part of the decomposition has a component in $A$, we define $a(g)\in A$ by $g=\bar{N}Ma(g)N$. 
We can see this directly for $SL(2,\R)$ as follows.
\begin{equation*}
\begin{pmatrix}a&b\\c&d\end{pmatrix}=\begin{pmatrix}1&0\\\frac{c}{a}&1\end{pmatrix}
\begin{pmatrix}a&0\\0&\frac{1}{a}\end{pmatrix}\begin{pmatrix}1&\frac{b}{a}\\0&1\end{pmatrix}\quad \textup{if }a\neq 0.
\end{equation*}

We describe the orbits of $L$ in $\bar{n}$. If we set
\begin{equation*}
X_q\equiv F_1+\cdots+F_q, \quad q=1,2,\cdots,n\quad \textup{and } X_0\equiv 0 
\end{equation*}
then by \cite{Muller96} and \cite{Kaneyuki98} the $L$-orbits in $\bar{n}$ are precisely
\begin{equation*}
\mathcal{O}_q=L(X_q), \quad q=0,1,2,\cdots,n.
\end{equation*}
We write $\mathcal{O}_q=L/S_q$, $S_q$ the stabilizer of $X_q$.
As $\textup{ad}(X_n):\fl\rightarrow\fn$ is onto, $\mathcal{O}_n$ is open in $\fn$.
Moreover, it is also dense and a semisimple symmetric space of rank $n$. 

The Iwasawa decomposition of $G$ with respect to $P$ is
\begin{equation*}
G=K\textup{exp}(\fm\cap\fs)AN.
\end{equation*}
We will write it as $g=\kappa(g)\mu(g)e^{H(g)}n(g)$. This decomposition for $SL(2,\R)$ is given by the identity:
\begin{equation*}
\begin{pmatrix}a&b\\c&d\end{pmatrix}
=\begin{pmatrix}\frac{a}{\sqrt{a^2+c^2}}&\frac{-c}{\sqrt{a^2+c^2}}\\\frac{c}{\sqrt{a^2+c^2}}&\frac{a}{\sqrt{a^2+c^2}}\end{pmatrix}
\begin{pmatrix}\sqrt{a^2+c^2}&0\\0&\frac{1}{\sqrt{a^2+c^2}}\end{pmatrix}\begin{pmatrix}1&\frac{ab+cd}{a^2+c^2}\\0&1\end{pmatrix}.
\end{equation*}
In particular, for $X=\begin{pmatrix}0&0\\x&0\end{pmatrix}$, $e^{H(\bar{n}_X)}=\begin{pmatrix}\sqrt{1+x^2}&0\\0&\frac{1}{\sqrt{1+x^2}}\end{pmatrix}$.
The following lemma is easily proved by the uniqueness of the decomposition.

\begin{Lem}\label{lem_kappa}
For $k\in K\cap L$ and $X\in \bar{\fn}$, $\kappa(\bar{n}_{k\cdot X})=k\,\kappa(\bar{n}_X)\,k^{-1}$.
\end{Lem}

We define $w\in \prod_{j=1}^nSL(2,\R)_j$ as $\displaystyle{w=\prod_{j=1}^n\begin{pmatrix}0&1\\-1&0\end{pmatrix}}$, which is in $K$, and satisfies $\textup{Ad}(w)\fn=\bar{\fn}$.
We also define functions on dense open subsets of $\bar{\fn}$ and $\fn$ by
\begin{align*}
\bar{\nabla}(X)&\equiv e^{\Lambda_0(\log(a(w\bar{n}_X)))},\quad X\in \bar{\fn},\quad \textup{and}\\
\nabla(Y)&\equiv \bar{\nabla}(\theta(Y)),\quad Y\in \fn.
\end{align*}
Then both $\bar{\nabla}$ and $\nabla$ are invariant under $K\cap L$.
For $G=GL(2n,\R)$, $\nabla=\bar{\nabla}=|\det|$.
The lemma below follows from the Bruhat decomposition of $SL(2,\R)$.
\begin{Lem}\label{lem_nabla}
For nonzero real numbers of $x_j$'s, $\displaystyle{\bar{\nabla}\left(\sum_{j=1}^nx_jF_j\right)=\prod_{j=1}^n|x_j|}$.
\end{Lem}

We describe a ``polar coordinates" expression for the Lebesgue measure on $\bar{\fn}$.
In \cite{Loos77} it is shown that the elements
\begin{equation*}
\left\{\sum_{j=1}^nx_jF_j:x_1>\cdots>x_n>0\right\}
\end{equation*}
give a complete set of orbit representatives for the action of $K\cap L$ on $\mathcal{O}_n$.
We write $\Omega$ for the cone
\begin{equation*}
\Omega=\left\{(x_1,\cdots,x_n):x_1>\cdots>x_n>0\right\}.
\end{equation*}
Each $dx_j$ denotes the Lebesgue measure on $\R$.

\begin{Prop}\textup{(\cite[Proposition 7.1.3]{Schlichtkrull84}, see also \cite[Proposition 1.2]{DvorskySahi03})}\label{prop_polar_coord}
Let $dX$ be the Lebesgue measures on $\bar{\fn}$. Then we have
\begin{equation*}
\int_{\bar{\fn}}f(X)dX\\=c\int_{K\cap L}\left[\int_{\Omega}f(k\cdot \sum_{j=1}^nx_jF_j)\prod_{j=1}^nx_j^e
\prod_{1\leq i<j\leq n}(x_i^2-x_j^2)^d\prod_{j=1}^ndx_j\right]dk.
\end{equation*}
\end{Prop}
The scalar $c$ in the above formula is independent of $f$ and depend on the normalization of the measure $dX$.

\section{Generalized principal series representations and the Lie algebra action}
For $s\in \C$, we may define a family of normalized principal series representations as induced pictures
\begin{align*}
\textup{Ind}_P^G(s)=\{f&:G\rightarrow \C:f \textup{ is smooth and}\\
&f(gman)=e^{-(s+\frac{m}{n})\Lambda_0(\log(a))}f(g), man\in P=MAN\}.
\end{align*}
The group $G$ acts by left translation:
\begin{equation*}
(g\cdot f)(g_1)=f(g^{-1}g_1).
\end{equation*}
Then the compact picture is obtained from the induced picture by restricting to $K$.
We set $\varphi=f|_K$, $f\in \textup{Ind}_P^G(s)$.
The compact picture may be written as
\begin{equation*}
C^{\infty}(K/M)=\{\varphi:K\rightarrow \C:\varphi \textup{ is smooth and }\varphi(km)=\varphi(k), k\in K\textup{ and } m\in M\}.
\end{equation*}
Note that the action of $G$ corresponding to the left translation does depend on $s$, 
but the function $\varphi$ in the compact picture is independent on $s$.

The noncompact picture is given by restricting the induced picture to $\bar{N}$.
For $f\in \textup{Ind}_P^G(s)$, we set $F_s(X)=f(\bar{n}_X)$. 
Then $\textup{Ind}_P^G(s)$ may be identified with
\begin{equation*}
I(s)=\{F_s\in C^{\infty}(\bar{\fn}):F_s(X)=f(\bar{n}_X),\,\,\textup{for some }f\in \textup{Ind}_P^G(s)\}.
\end{equation*}
By the Bruhat decomposition, $g\in\bar{N}P$ acts by
\begin{equation*}
g\cdot F_s(X)=e^{-(s+\frac{m}{n})\Lambda_0}(a(g^{-1}\bar{n}_X))F_s(\log(\bar{N}(g^{-1}\bar{n}_X))).
\end{equation*}
In particular
\begin{align*}
&(\ell\cdot F_s)(X_1)=e^{-(s+\frac{m}{n})\Lambda_0}(a(\ell^{-1}))F_s(\ell^{-1}\cdot X_1)\quad\textup{and}\\
&(\bar{n}_X\cdot F_s)(X_1)=F_s(X_1-X).
\end{align*}

We may express functions in the noncompact picture in terms of functions in the compact picture.
By the Iwasawa decomposition, $F_s(X)=e^{-(s+\frac{m}{n})\Lambda_0(H(\bar{n}_X))}\varphi(\kappa(\bar{n}_X))$.
If $\varphi\equiv 1$, then $F_s$ is called the spherical function, denoted by $h_s$.
That is,
\begin{equation*}
h_s(X)=e^{-(s+\frac{m}{n})\Lambda_0(H(\bar{n}_X))}.
\end{equation*}
Note that $h_s$ is invariant under $K$.
The lemma below follows from the Iwasawa decomposition for $SL(2,\R)$ and Lemma \ref{lem_kappa}.
\begin{Lem}\label{lem_F_s_phi}
Let $X=k\cdot \sum_{j=1}^nx_jF_j$, for $k\in K\cap L$.
Then the spherical function is
\begin{equation*}
h_s(X)=\frac{1}{\prod_{j=1}^n(1+x_j^2)^{\frac{s+\frac{m}{n}}{2}}}.
\end{equation*}
Therefore, $F_s\in I(s)$ can be expressed as
\begin{equation*}
F_s(X)=\frac{(k^{-1}\cdot\varphi)\left(\kappa\left(\sum_{j=1}^nx_jF_j\right)\right)}{\prod_{j=1}^n(1+x_j^2)^{\frac{s+\frac{m}{n}}{2}}}.
\end{equation*}
\end{Lem}

We describe the action by elements of the Lie subalgebras $\mathfrak{sl}(2,\R)_j$ of $\fg$.
The Lie algebra action will play an important role in proving the main results.
For the group representation on $I(s)$, the correspond Lie algebra representation is given by the formula, for $Y\in \fg$,
\begin{align*}
Y\cdot F_s(X)&=\frac{d}{dt}\Big|_{t=0}\textup{exp}(tY)\cdot f(\bar{n}_X)\\
&=\frac{d}{dt}\Big|_{t=0}e^{-(s+\frac{m}{n})\Lambda_0}(\ell(\textup{exp}(-tY)\bar{n}_X))F_s(\log(\bar{N}(\textup{exp}(-tY)\bar{n}_X)))
\end{align*} 

\begin{Lem}\label{lem_Lie_algebra_action}
Let $X=\sum_{j=1}^nx_jF_j$. The action of $\mathfrak{sl}(2,\R)_j$ on $I(s)$ is given by
\begin{enumerate}
\item $E_k\cdot F_s(X)=(s+\frac{m}{n})x_kF_s(X)+x_k^2\frac{\partial F_s}{\partial x_k}(X)$,
\item $H_k\cdot F_s(X)=(s+\frac{m}{n})F_s(X)+2x_k\frac{\partial F_s}{\partial x_k}(X)$, and
\item $F_k\cdot F_s(X)=-\frac{\partial F_s}{\partial x_k}(X)$.
\end{enumerate}
\end{Lem}
\begin{proof}
For sufficiently small values of $t$, we may assume that $1-tx_k>0$. 
In $\prod_{j=1}^n\textup{SL}(2,\R)_j$, by the Bruhat decomposition, we have
\begin{align*}
\textup{exp}(-tE_k)\textup{exp}(X)&=\textup{exp}\left(\sum_{j\neq k}x_jF_j+\frac{x_k}{1-tx_k}F_k\right)
\begin{pmatrix}1-tx_k&0\\0&\frac{1}{1-tx_k}\end{pmatrix}_k\begin{pmatrix}1&\frac{-t}{1-tx_k}\\0&1\end{pmatrix}_k\,,\\
\textup{exp}(-tH_k)\textup{exp}(X)&=\textup{exp}\left(\sum_{j\neq k}x_jF_j+x_ke^{2t}F_k\right)
\textup{exp}\left(-tH_k\right)\,,\textup{and}\\
\textup{exp}(-tF_k)\textup{exp}(X)&=\textup{exp}\left(\sum_{j\neq k}x_jF_j+(x_k-t)F_k\right)\,.
\end{align*}
This proves the lemma.
\end{proof}

\section{Zeta distributions and integrals for functions in $I(s)$}

The functions $\bar{\nabla}(X)^t$ and $\nabla(Y)^t$ are locally $L^1$ functions 
for $\textup{Re}(t)>-(e+1)$ and they define tempered distributions by the integrals
\begin{align*}
\bar{\Z}(f,t)&=\int_{\bar{\fn}}f(X)\bar{\nabla}(X)^tdX, \textup{ for }f\in \mathcal{S}(\bar{\fn})\textup{ and}\\
\Z(h,t)&=\int_{\fn}h(Y)\nabla(Y)^tdY, \textup{ for }h\in \mathcal{S}(\fn).
\end{align*}
Here $\mathcal{S}(\bar{\fn})$ (resp.~ $\mathcal{S}(\fn)$) denotes the space of Schwartz functions on $\bar{\fn}$ (resp.~ $\fn$).
Note that in the range $\textup{Re}(t)>-(e+1)$ both expressions are complex analytic functions of $t$.
It is also well-known that there is a meromorphic continuation to all of $\C$.
This can be viewed by a general result of \cite{Bernstein72} as follows.
Both $\bar{\nabla}^2$ and $\nabla^2$ are polynomials 
by \cite[Lemma 2.10]{BarchiniSepanskiZierau06}, for example. 
There is a polynomial $b(t)$ so that
\begin{equation*}
\bar{\nabla}(\partial_X)^2\bar{\nabla}(X)^t=b(t)\bar{\nabla}(X)^{t-2}\textup{ and }
\nabla(\partial_Y)^2\nabla(Y)^t=b(t)\nabla(Y)^{t-2}.
\end{equation*}
In particular, for $b_k(t)=b(t)b(t-2)b(t-4)\cdots b(t-2(k-1))$,
\begin{equation*}
\bar{\nabla}(\partial_X)^{2k}\bar{\nabla}(X)^t=b_k(t)\bar{\nabla}(X)^{t-2k}\textup{ and }
\nabla(\partial_Y)^{2k}\nabla(Y)^t=b_k(t)\nabla(Y)^{t-2k}.
\end{equation*}
We integrate by parts to get
\begin{equation*}\label{eq:b_k}
\bar{\Z}(\bar{\nabla}(\partial_X)^{2k}f, t)=b_k(t)\bar{\Z}(f,t-2k)\textup{ and }
\Z(\nabla(\partial_Y)^{2k}h, t)=b_k(t)\Z(h,t-2k)
\end{equation*}
for sufficiently large $\Re(t)$.
Since the left hand side is analytic for $\Re(t)>-(e+1)$, both $\bar{\Z}(f,t)$ and $\Z(h,t)$
continue to meromorphic functions on $\Re(t)>-(e+1)-2k$ for any $k$.

We now turn to the Fourier transformation and functional equation.
Let $B$ denote the Killing form of $\fg$.
Then $B$ gives a nondegenerate pairing between $\bar{\fn}$ and $\fn$.
Since the form $-B(\cdot\,,\,\theta(\cdot))$ is positive definite on $\fg$, we may define inner products
\begin{align*}
\langle X_1,X_2\rangle &= -\frac{n}{4m}B(X_1,\theta(X_2))\textup{ on }\bar{\fn}\textup{ and}\\
\langle Y_1,Y_2\rangle &= -\frac{n}{4m}B(Y_1,\theta(Y_2))\textup{ on }\fn.
\end{align*}
For example, if $G=GL(2n,\R)$, then $\langle X_1,X_2\rangle$ is the trace form on $\bar{\fn}$.
Define the Fourier transforms by 
\begin{align*}
\widehat{f}(Y)&=\int_{\bar{\fn}}f(X)e^{-2\pi i\langle X,Y\rangle }dX,\textup{ for } f\in L^1(\bar{\fn})\textup{ and}\\
\widehat{h}(X)&=\int_{\fn}h(Y)e^{-2\pi i\langle X,Y\rangle }dY,\textup{ for } h\in L^1(\fn).
\end{align*}
We may regard the Fourier transform $\widehat{f}$ on $\bar{\fn}$ (resp. $\widehat{h}$ on $\fn$) 
as a function on $\fn$ (resp. $\bar{\fn}$).
There is a functional equation relating the two distributions via the Fourier transform.
We let $\Gamma$ denote the gamma function on $\C$.

\begin{Thm}\textup{(\cite{Muller})}
Let $t\in \C$ and $f\in \mathcal{S}(\bar{\fn})$. As meromorphic functions
\begin{equation}\label{eq:FT_S}
\frac{\pi^{\frac{nt}{2}}}{\Gamma_n(t)}\bar{\Z}(f,t-\frac{m}{n})=\frac{\pi^{\frac{n}{2}(-t+\frac{m}{n})}}{\Gamma_n(-t+\frac{m}{n})}\Z(\widehat{f},-t),
\end{equation}
where
\begin{equation*}
\Gamma_n(t)=\prod_{j=0}^{n-1}\Gamma\left(\frac{t-jd}{2}\right).
\end{equation*}
\end{Thm}

To show that the functional equation (\ref{eq:FT_S}) holds for functions in $I(s)$, we consider the integral
\begin{align*}
\bar{\Z}(F_s,t)=\int_{\bar{\fn}}F_s(X)\bar{\nabla}(X)^tdX, \textup{ for }F_s\in I(s).
\end{align*}

\begin{Lem}\label{lem:z_conv}
$\bar{\Z}(F_s,t)$ converges absolutely on $-(e+1)<\textup{Re}(t)<\textup{Re}(s)-d(n-1)$.
\end{Lem}
\begin{proof}
Combining Lemma \ref{lem_nabla}, Proposition \ref{prop_polar_coord}, and Lemma \ref{lem_F_s_phi}, 
we get
\begin{equation}\label{eq:z_polar}
\bar{\Z}(F_s,t)=c\int_{K\cap L}\left[\int_{\Omega}
\frac{(k^{-1}\cdot\varphi)\left(\kappa\left(\sum_{j=1}^nx_jF_j\right)\right)}{\prod_{j=1}^n(1+x_j^2)^{\frac{s+\frac{m}{n}}{2}}}
\prod_{j=1}^nx_j^{t+e}\prod_{1\leq i<j\leq n}(x_i^2-x_j^2)^d\prod_{j=1}^ndx_j\right]dk\,.
\end{equation}
Expanding $(x_i^2-x_j^2)^d$, we can write the integrand as a sum of terms
\begin{equation*}
(k^{-1}\cdot\varphi)\left(\kappa\left(\sum_{j=1}^nx_jF_j\right)\right)\prod_{j=1}^n\frac{x_j^{t+e+k}}{(1+x_j^2)^{\frac{s+\frac{m}{n}}{2}}},
\textup{ where }k=0, 2, \cdots, 2d(n-1).
\end{equation*}
Each of these integrals is bounded by a product of one variable integrals of the form
\begin{equation*}
\int_0^{\infty}\frac{x_j^{\textup{\Re}(t)+e+k}}{(1+x_j^2)^{\frac{\textup{\Re}(s)+\frac{m}{n}}{2}}}dx_j,
\end{equation*}
which converge if $-1<\textup{\Re}(t)+e+k<\textup{\Re}(s)+\frac{m}{n}-1$. This proves the lemma.
\end{proof}

\begin{Cor}
$F_s\in L^1(\bar{\fn})$ for $\Re(s)>d(n-1)$ and $F_s\in L^2(\bar{\fn})$ for $\Re(s)>-\frac{e+1}{2}$.
\end{Cor}
\begin{proof}
If $t=0$ in Lemma \ref{lem:z_conv}, we get the condition on $s$ so that $F_s\in L^1(\bar{\fn})$.
Note also that $F_s\in L^2(\bar{\fn})$ if and only if $h_s^2\in L^1(\bar{\fn})$.
This is the case when $e+k<2\left(s+\frac{m}{n}\right)-1$ for all $k=0, 2, \cdots, 2d(n-1)$.
\end{proof}

In the same way we continue $\bar{\Z}(f,t)$, for $f\in \mathcal{S}(\bar{\fn})$, to a meromorphic function in $t\in \C$, we continue $\bar{\Z}(F_s,t)$, for $F_s\in I(s)$, meromorphically.

\begin{Lem}\label{lem:mero_t}
$\bar{\Z}(F_s,t)$ can be extended meromorphically to the range
\begin{equation*}
\bigcup_{k\in \Z_{\geq 0}}\left\{(s,t)\in \C^2:
-(e+1)-2k<\textup{Re}(t)<\textup{Re}(s)-d(n-1)-2k\right\}
\end{equation*}
\end{Lem}
\begin{proof}
Note that $\bar{\nabla}^{\,t}(X)$ vanishes on the boundary $\bar{\nabla}(X)=0$ for $\Re(t)>0$ and
$F_s(X)$ vanishes at infinity for $\Re(s)>-\frac{m}{n}$.
Also, since each $\partial^{\alpha}$, $|\alpha|=1$, acts as an element 
in the enveloping algebra $\mathcal{U}(\fg)$, 
$\partial^{\alpha}(F_s)\in I(s)$ for any multi-index $\alpha$. 
So is $\bar{\nabla}(\partial_X)^{2k}(F_s)$ for any $k\in \Z_{\geq 0}$.
Therefore, for $F_s\in I(s)$
\begin{equation}\label{eq:b_k}
\bar{\Z}(\bar{\nabla}(\partial_X)^{2k}F_s, t)=b_k(t)\bar{\Z}(F_s,t-2k)
\end{equation}
is to hold as convergent integrals 
when $2k<\Re(t)<\Re(s)-d(n-1)$ by Lemma \ref{lem:z_conv}.
Since the left hand side is analytic on $-(e+1)<\textup{Re}(t)<\textup{Re}(s)-d(n-1)$, 
$\bar{\Z}(F_s,t)$ continues to a meromorphic function 
on $-(e+1)-2k<\textup{Re}(t)<\textup{Re}(s)-d(n-1)-2k$ for all $k\in \Z_{\geq 0}$.
\end{proof}

Much of background materials in sections $2$, $3$, and $4$ comes from \cite{BarchiniSepanskiZierau06}, \cite{DvorskySahi03}, and  \cite{KosSah93}.

\section{Proof of the main results}
In this section we prove the following theorems.
\begin{Thm}\label{thm:mero_non}
Let $F_s\in I(s)$. Then the family of integrals $\bar{\Z}(F_s,t)$ is complex analytic on $-(e+1)<\textup{Re}(t)<\textup{Re}(s)-d(n-1)$
and has a meromorphic continuation to all of $\C^2$.
\end{Thm}

\begin{Thm}\label{thm:extended_F_E}
Let $F_s\in I(s)$. Then the functional equation
\begin{equation}\label{eq:FT_non}
\frac{\pi^{\frac{nt}{2}}}{\Gamma_n(t)}\bar{\Z}(F_s,t-\frac{m}{n})
=\frac{\pi^{\frac{n}{2}(-t+\frac{m}{n})}}{\Gamma_n(-t+\frac{m}{n})}\Z(\widehat{F_s},-t)
\end{equation}
holds as meromorphic functions in $(s,t)\in \C^2$.
\end{Thm}
We will use the integral formula Proposition \ref{prop_polar_coord} for the polar coordinates 
to reduce the integral $\Z(F_s,t)$ to the integrals over a noncompact radial set and a compact set.
The main part of the proof of Theorem \ref{thm:mero_non} is to have a meromorphic continuation of the integral over the noncompact set. 
We use the formulas of the action by the Lie subalgebras $\mathfrak{sl}(2,\R)_j$. 
These formulas give us appropriate differential operators that we will use to extend the defining range of $t$ and $s$. 
On the other hand, Theorem \ref{thm:extended_F_E} can be proved using \cite[Lemma 4.34]{Kimura03} and Theorem \ref{thm:mero_non}.

\subsection{Proof of Theorem \ref{thm:mero_non}}
$(\R^+)^n$ contains $n!$ disjoint cones of the form
\begin{equation*}
\sigma\Omega=\{{\bf x}=(x_1,\ldots,x_n)\,:\, x_{\sigma(1)}>\cdots>x_{\sigma(n)}>0\}.
\end{equation*}
Here $S_n$ is the symmetric group on $n$ letters.
Note also that the disjoint union $\cup_{\sigma\in S_n}\sigma\Omega$ is dense and open in $(\R^+)^n$.
Then integrand of (\ref{eq:z_polar}) is invariant under the action of permuting $x_j$'s 
except the possible negative sign from $\prod_{1\leq i<j\leq n}(x_i^2-x_j^2)^d$.
We define the sign function on $(\R^+)^N$ as
\begin{equation*}
\varepsilon(x_1,\cdots,x_N)=\textup{sign}\left(\prod_{i<j}(x_i-x_j)\right).
\end{equation*}
Then we can rewrite (\ref{eq:z_polar}) as
\begin{equation*}
\bar{\Z}(F_s,t)=\frac{c}{n!}\int_M\left[\int_{(\R^+)^n}
\frac{(m^{-1}\cdot\varphi)\left(\kappa\left(\sum_{j=1}^nx_jF_j\right)\right)}{\prod_{j=1}^n(1+x_j^2)^{\frac{s+\frac{m}{n}}{2}}}
\prod_{j=1}^nx_j^{t+e}\prod_{1\leq i<j\leq n}(x_i^2-x_j^2)^d\,\varepsilon({\bf x})^d\prod_{j=1}^ndx_j\right]dm\,.
\end{equation*}
By expanding $(x_i^2-x_j^2)^d$, the integral $\bar{\Z}(F_s,t)$ is a finite linear combination of integrals of the form
\begin{equation}\label{eq:z_polar_sign}
\int_M\left[\int_{(\R^+)^n}
\frac{(m^{-1}\cdot\varphi)\left(\kappa\left(\sum_{j=1}^nx_jF_j\right)\right)}{\prod_{j=1}^n(1+x_j^2)^{\frac{s+\frac{m}{n}}{2}}}
\prod_{j=1}^nx_j^{t+c_j}\varepsilon({\bf x})^d\prod_{j=1}^ndx_j\right]dm\,,
\end{equation}
where $e\leq c_j\leq e+2d(n-1)$.
As the first step to prove Theorem \ref{thm:mero_non}, 
we assume that $m=I$ in (\ref{eq:z_polar_sign}) and consider the integral over the noncompact set $(\R^+)^n$.
To make an induction argument work, we need to consider a more general integral than (\ref{eq:z_polar_sign}).
We introduce some notations.
Let ${\bf a}=(a_1,\cdots,a_n)\in \N^n$.
Let ${\bf b}=(b_1,\cdots,b_n)$, ${\bf c}=(c_1,\cdots,c_n)\in (\R^+)^n$.
Let $N$ be a positive integer with $N\leqslant n$.
Let ${\bf p}=(p_1,\cdots,p_N)\in \Z^N$ with $1\leqslant p_1<\cdots<p_N\leqslant n$.
Let ${\bf q}=(q_1,\cdots,q_n)\in \Z^n$ with $1\leqslant q_j \leqslant n$ for all $j=1,\cdots,n$.
For ${\bf x}=(x_1,\cdots,x_n)\in (\R^+)^n$, we define ${\bf x}_{{\bf p}}=(x_{p_1},\cdots,x_{p_N})\in (\R^+)^N$ and 
$[{\bf x}_{{\bf q}}]=\kappa\left(\sum_{j=1}^nx_{q_j}F_j\right)\in K$.

\begin{Prop}\label{prop:mero_cont_noncompact}
Let $\varphi\in C^{\infty}\left(K/K\cap M\right)$. Then the integral
\begin{align}
\T_n \left(s{\bf a}+{\bf b},t{\bf a}+{\bf c},{\bf p},\varphi[{\bf x}_{\bf q}]\right)
=\int_{(\R^+)^n} \prod_{j=1}^n\frac{x_j^{a_jt+c_j}}{(1+x_j^2)^{\frac{a_js+b_j+\frac{m}{n}}{2}}}\cdot
\varepsilon({\bf x}_{{\bf p}})^d\,\varphi[{\bf x}_{{\bf q}}] dx_n\cdots dx_1 
\end{align}
converges absolutely and is complex analytic on 
\begin{equation}\label{convergence_range_T_D_0}
\mathcal{D}_0=\left\{(s,t)\in \C^2 \,:\, 
\max_{j}\left\{\frac{-1-c_j}{a_j}\right\}<\textup{Re}(t)<\textup{Re}(s)+\min_{j}\left\{\frac{b_j+\frac{m}{n}-1-c_j}{a_j}\right\}\right\}.
\end{equation}
Moreover, it has a meromorphic continuation to all of $(s,t)\in \C^2$. 
\end{Prop}
The proof of the convergence range $\mathcal{D}_0$ is analogous to the proof of Lemma \ref{lem:z_conv}.
The main ingredient to prove Proposition \ref{prop:mero_cont_noncompact} is to apply some differential operators from the Lie algebra action.
For $j=1,\cdots, n$, we define differential operators $D_{s,t}^j$ by the action of the Lie algebra element
\begin{equation*}
\begin{pmatrix}0&t\\s-t+n-2&0\end{pmatrix}_j=tE_j+(s-t+n-2)F_j.
\end{equation*}
Then we have
\begin{equation*}
D_{s,t}^j=\left(s+\frac{m}{n}\right)t\,x_j+t\left(1+x_j^2\right)\frac{\partial}{\partial x_j}-\left(s+\frac{m}{n}-2\right)\frac{\partial}{\partial x_j}\,, 
\end{equation*}
by Lemma \ref{lem_Lie_algebra_action}. Let
\begin{equation*}
{D_{s,t}^j}^{\dagger}=\left(s+\frac{m}{n}\right)t\,x_j-t\frac{\partial}{\partial x_j}\circ\left(1+x_j^2\right)
+\left(s+\frac{m}{n}-2\right)\frac{\partial}{\partial x_j}. 
\end{equation*}
Then ${D_{s,t}^j}^{\dagger}$ is the formal adjoint of $D_{s,t}^j$ and we have
\begin{equation}\label{eq:dagger}
{D_{s,t}^j}^{\dagger}\,x_j^t=\left(s+\frac{m}{n}-t-2\right)t\,x_j^{t-1}\left(1+x_j^2\right).
\end{equation}
We shall denote the constant $\left(s+\frac{m}{n}-t-2\right)t$ by $(s:t)$.
For a function $f$ on $(a,b)$, we set
\begin{align*}
&[f(x)]_{x\rightarrow a^+}^{x\rightarrow b^-}:=\lim_{x\rightarrow b^-}f(x)-\lim_{x\rightarrow a^+}f(x)\quad\textup{and}\\
&f(x)|_{x\rightarrow b^-}:=\lim_{x\rightarrow b^-}f(x)
\end{align*}
if the limits exist.
By integration by parts, it is easily proved
\begin{equation}\label{eq:parts_adj}
\int_a^bf(x_l)\left(D_{s,t}^l\,g(x_l)\right)dx_l=\int_a^b\left({D_{s,t}^l}^{\dagger}\,f(x_l)\right)g(x_l)dx_l
+\left[\left\{t\left(1+x_l^2\right)-\left(s+\frac{m}{n}-2\right)\right\}f(x_l)g(x_l)\right]_{x_l\rightarrow a^+}^{x_l\rightarrow b^-}
\end{equation}
for all $(s,t)\in \C^2$ at which each term is defined.
We define $E^l(s,t,{\bf p},\varphi[{\bf x}_{{\bf q}}])$ as $0$ if $d$ is even or $d$ is odd and $l\neq p_i$ for all $i=1,\cdots,N$.
If $d$ is odd and $l=p_i$ for some $i=1,\cdots,N$, then
\begin{equation*}
E^l(s,t,{\bf p},\varphi[{\bf x}_{\bf q}])=
\sum_{_{(p_i\neq l)}^{\,\,i=1}}^N \frac{2\cdot\left\{t(1+x_{p_i}^2)-(s+\frac{m}{n}-2)\right\}x_{p_i}^t}{(1+x_{p_i}^2)^{\frac{s+\frac{m}{n}}{2}}}\,\,
\left( \varepsilon({\bf x}_{{\bf p}})\varphi[{\bf x}_{{\bf q}}]\right)\Big|_{x_l\rightarrow x_{p_i}^-}\,.
\end{equation*}
Note that $E^l$ does not contain the variable $x_l$.

\begin{Lem}\label{lem_one_operator}
For $(s,t)\in \C^2$ with $0<\Re(t)<\Re(s)+\frac{m}{n}-2$, 
\begin{equation*}
\int_0^{\infty} x_l^t\,\,\varepsilon({\bf x}_{{\bf p}})^dD_{s,t}^l\left(\frac{\varphi[{\bf x}_{{\bf q}}]}
{(1+x_l^2)^{\frac{s+\frac{m}{n}}{2}}}\right)dx_l
=(s:t)\int_0^{\infty} x_l^{t-1}\,\,\varepsilon({\bf x}_{{\bf p}})^d\left(\frac{\varphi[{\bf x}_{{\bf q}}]}
{(1+x_l^2)^{\frac{s-2+\frac{m}{n}}{2}}}\right) dx_l+E^l(s,t;{\bf p};\varphi[{\bf x}_{\bf q}]). 
\end{equation*}
\end{Lem} 
\begin{proof}
We apply the integration by parts on the subintervals given by the partition $\{x_1,\cdots, \widehat{x_l},\cdots,x_n\}$ on $(0,\infty)$.
We note that 
\begin{equation}\label{eq:boundary_0}
\left[\left\{t\left(1+x_l^2\right)-\left(s+\frac{m}{n}-2\right)\right\}x_l^t\,\,\varepsilon({\bf x}_{{\bf p}})^d\left(
\frac{\varphi[{\bf x}_{{\bf q}}]}{(1+x_l^2)^{\frac{s+\frac{m}{n}}{2}}}\right) \right]_{x_l\rightarrow 0^+}^{x_l\rightarrow {\infty}}=0
\end{equation}
on  $0<\Re(t)<\Re(s)+\frac{m}{n}-2$.
This proves the lemma.
\end{proof}

The equation in Lemma \ref{lem_one_operator} can be used 
to compute an integral whose integrand contains several differential operators of the form $D_{s,t}^l$.
Let $a$ be a positive integer. We define
\begin{equation*}
\widetilde{D}_{s,t}^l=D_{s,t}^l\circ\frac{1}{(1+x_l^2)}\,.
\end{equation*}
Define $\psi_a[{\bf x}]=\varphi[{\bf x}_{\bf q}]$ and $\psi_{a-1},\cdots,\psi_0$ inductively by the equations 
\begin{equation*}
\frac{\psi_k[{\bf x}]}{(1+x_l^2)^{\frac{s-2k+\frac{m}{n}}{2}}}
=D_{s-2k,t-k}^l\frac{\psi_{k+1}[{\bf x}]}{(1+x_l^2)^{\frac{s-2k+\frac{m}{n}}{2}}}
\quad\textup{for  }k=0,1,\cdots,a-1.
\end{equation*}
Then we have
\begin{equation*}
\frac{\psi_0[{\bf x}]}{(1+x_l^2)^{\frac{s+\frac{m}{n}}{2}}}
=\left(\prod_{k=0}^{a-1}\widetilde{D}_{s-2k,t-k}^l\right)\frac{\varphi[{\bf x}_{\bf q}]}{(1+x_l^2)^{\frac{s-2a+\frac{m}{n}}{2}}}\,.
\end{equation*}
Since $D_{s-2k,t-k}$ acts as an element in the Lie algebra on $I(s-2k)$,
it preserve the representation space $I(s-2k)$.
Therefore, the condition that $\varphi\in C^{\infty}(K/K\cap M)$ implies $\psi_k\in C^{\infty}(K/K\cap M)$ for all $k=0,\cdots, a-1$.
We let $\gamma_0(s:t)=1$. Define
\begin{align*}
&\gamma_k(s:t)=\prod_{r=0}^{k-1}(s-2r:t-r)\textup{ for }k=1,\cdots,a-1\quad\textup{and}\\ 
&\E^{l,a}(s,t,{\bf p},\varphi[{\bf x}_{{\bf q}}])=\sum_{k=0}^{a-1}\gamma_k(s:t)E^l(s-2k,t-k,{\bf p},\psi_{k+1}[{\bf x}])\,.
\end{align*}
We use the induction on $k$ (for $k=0,\cdots,a$) to get the following statement:
\begin{align*}
\int_0^{\infty} x_l^t\,\,\varepsilon({\bf x}_{{\bf p}})^d\frac{\psi_0[{\bf x}]}{(1+x_l^2)^{\frac{s+\frac{m}{n}}{2}}}dx_l
=&\gamma_k(s:t)\int_0^{\infty} x_l^{t-k}\,\,\varepsilon({\bf x}_{{\bf p}})^d
\frac{\psi_k[{\bf x}]}{(1+x_l^2)^{\frac{s-2k+\frac{m}{n}}{2}}} dx_l\\
&+\sum_{r=0}^{k-1}\gamma_r(s:t)E^l(s-2r,t-r,{\bf p},\psi_{r+1}[{\bf x}])\,.
\end{align*}
Then the lemma below follows when $k=a$. The convergence range is from (\ref{eq:boundary_0}).
\begin{Lem}\label{lem_operators}
For $(s,t)\in \C^2$ with $a-1<\textup{Re}(t)<\textup{Re}(s)-(a-1)+\frac{m}{n}-1$, 
\begin{equation*}
\int_0^{\infty} x_l^t\,\varepsilon({\bf x}_{{\bf p}})^d\frac{\psi_0[{\bf x}]}{(1+x_l^2)^{\frac{s+\frac{m}{n}}{2}}}dx_l
=\gamma_a(s:t)\int_0^{\infty} x_l^{t-a}\,\varepsilon({\bf x}_{{\bf p}})^d
\frac{\varphi[{\bf x}_{{\bf q}}]}{(1+x_l^2)^{\frac{s-2a+\frac{m}{n}}{2}}} dx_l+\E^{l,a}(s,t,{\bf p},\varphi[{\bf x}_{{\bf q}}])\,.
\end{equation*}
\end{Lem}   

Recall ${\bf a}$, ${\bf b}$, ${\bf c}$, $N$, ${\bf p}$ and $\varphi[{\bf x}_{\bf q}]$ as in Proposition \ref{prop:mero_cont_noncompact}.
Let ${\bf s}=s{\bf a}+{\bf b}=(s_1,\cdots,s_n)$ and ${\bf t}=t{\bf a}+{\bf c}=(t_1,\cdots,t_n)$.
Define $\eta_n[{\bf x}]=\varphi[{\bf x}_{\bf q}]$ and $\eta_{n-1},\cdots,\eta_0$ by the equations inductively
\begin{equation*}
\frac{\eta_{k-1}[{\bf x}]}{(1+x_k^2)^{\frac{s_k+\frac{m}{n}}{2}}}=\left(\prod_{r=0}^{a_k-1}\widetilde{D}_{s_k-2r,t_k-r}^k\right)
\frac{\eta_k[{\bf x}]}{(1+x_k^2)^{\frac{s_k-2a_k+\frac{m}{n}}{2}}}\quad\textup{for  }n\geqslant k\geqslant 1.
\end{equation*}
Then we have
\begin{equation*}
\frac{\eta_0[{\bf x}]}{\prod_{k=1}^n(1+x_k^2)^{\frac{s_k+\frac{m}{n}}{2}}}
=\prod_{k=1}^n\left(\prod_{r=0}^{a_k-1}\widetilde{D}_{s_k-2r,t_k-r}^k\right)\frac{\varphi[{\bf x}_{\bf q}]}
{\prod_{k=1}^n(1+x_k^2)^{\frac{s_k-2a_k+\frac{m}{n}}{2}}}\,.
\end{equation*}
Since $\varphi\in C^{\infty}(K/K\cap M)$, $\eta_k\in C^{\infty}(K/K\cap M)$ for all $k=0,\cdots, n-1$.
We let $\gamma_0({\bf s}:{\bf t})=1$. Define 
\begin{align*}
&\gamma_k({\bf s}:{\bf t})=\prod_{r=1}^k\gamma_{a_r}(s_r:t_r)\textup{ for }k=1,\cdots,n-1\quad\textup{and}\\
&\E_{n-1}({\bf s},{\bf t},{\bf p},\varphi[{\bf x}_{\bf q}])\\
&=\sum_{k=0}^{n-1}\gamma_k({\bf s}:{\bf t})
\int_{(\R^+)^{n-1}}\prod_{j<k}\frac{x_j^{t_j-a_j}}{(1+x_j^2)^{\frac{s_j-2a_j+\frac{m}{n}}{2}}}
\prod_{j>k}\frac{x_j^{t_j}}{(1+x_j^2)^{\frac{s_j+\frac{m}{n}}{2}}}\cdot \E^{k,{a_k}}(s_k,t_k,{\bf p},\eta_k[{\bf x}])\prod_{j\neq k}dx_j\,.
\end{align*}
We use induction on $k$ (for $k=0,\cdots,n$) to get the following statement:
\begin{align*}
&\T_n \left({\bf s},{\bf t};{\bf p};\eta_0({\bf x})\right)\\
&=\gamma_k({\bf s}:{\bf t}) \int_{(\R^+)^n}\prod_{j\leqslant k}\frac{x_j^{t_j-a_j}}{(1+x_j^2)^{\frac{s_j-2a_j+\frac{m}{n}}{2}}}
\prod_{j>k}\frac{x_j^{t_j}}{(1+x_j^2)^{\frac{s_j+\frac{m}{n}}{2}}}\cdot\varepsilon({\bf x}_{{\bf p}})^d\cdot\eta_k({\bf x})\prod_{j=1}^ndx_j\\
&\quad +\sum_{r=0}^{k-1}\gamma_r({\bf s}:{\bf t}) \int_{(\R^+)^{n-1}}\prod_{j<r}\frac{x_j^{t_j-a_j}}{(1+x_j^2)^{\frac{s_j-2a_j+\frac{m}{n}}{2}}}
\prod_{j>r}\frac{x_j^{t_j}}{(1+x_j^2)^{\frac{s_j+\frac{m}{n}}{2}}}
\cdot \E^{r,{a_r}}(s_r,t_r,{\bf p},\eta_r[{\bf x}])\prod_{j\neq r}dx_j.
\end{align*}
Then the lemma below follows when $k=n$.
\begin{Lem}
For sufficiently large $\alpha$ and small $\beta$, we have
\begin{equation}\label{eq:FE_T_n}
\T_n \left({\bf s},{\bf t},{\bf p},\eta_0[{\bf x}]\right)
=\gamma_n({\bf s}:{\bf t})\T_n \left({\bf s}-2{\bf a},{\bf t}-{\bf a},{\bf p},\varphi[{\bf x}_{\bf q}]\right)
+\E_{n-1}({\bf s},{\bf t},{\bf p},\varphi[{\bf x}_{\bf q}])
\end{equation}
on $\alpha<\textup{Re}(t)<\textup{Re}(s)+\beta$.
In particular, if $d$ is even, then we have
\begin{equation*}
\T_n \left({\bf s},{\bf t},{\bf p},\eta_0[{\bf x}]\right)
=\gamma_n({\bf s}:{\bf t})\T_n \left({\bf s}-2{\bf a},{\bf t}-{\bf a},{\bf p},\varphi[{\bf x}_{\bf q}]\right)
\end{equation*}
\end{Lem}   

\begin{proof}[Proof of Proposition \ref{prop:mero_cont_noncompact}]
The analyticity of the integral $\T_n$ is a standard application of Morera's Theorem as follows.
The continuity follows from Lebesgue dominated convergence Theorem.
For any simple closed curve $C\in \mathcal{D}_0$ in $t\in \C$ (resp. $s\in \C$) for fixed $s$ (resp. $t$),
the integral over $C$ of the integrand of $\T_n$ is $0$ by Cauchy's integral formula.
By Fubini's Theorem, we show the integral over $C$ of the integral $\T_n$ is $0$. 

We use the induction on $n$ to prove the meromorphic continuation part.
For $n=1$, $\E_{n-1}({\bf s},{\bf t},{\bf p},\varphi[{\bf x}_{\bf q}])=0$ because $\varepsilon(x)=1$.
This case can be proved by the equation (\ref{eq:FE_T_n}). 
We suppose that $\T_{n-1}({\bf s},{\bf t},{\bf p},\varphi[{\bf x}_{\bf q}])$ has a meromorphic continuation to all of $(s,t)$ in $\C^2$
for all choices of ${\bf s}, {\bf t}, {\bf p}, {\bf q}$, and $\varphi\in C^{\infty}(K/K\cap M)$ .
Then $\E_{n-1}({\bf s},{\bf t},{\bf p},\varphi[{\bf x}_{\bf q}])$ can be extended to a meromorphic function in $(s,t)\in \C^2$ 
because it is a finite sum of integrals of the form $\T_{n-1}$. 
We rewrite the equation (\ref{eq:FE_T_n}) of the following form:
\begin{equation}\label{eq:FE_T_n_rewrite}
\T_n \left({\bf s},{\bf t},{\bf p},\varphi[{\bf x}_{\bf q}]\right)
=\frac{\T_n \left({\bf s}+2{\bf a},{\bf t}+{\bf a},{\bf p},\eta_0[{\bf x}]\right)
-\E_{n-1}({\bf s}+2{\bf a},{\bf t}+{\bf a},{\bf p},\varphi[{\bf x}_{\bf q}])}
{\gamma_n({\bf s}+2{\bf a}:{\bf t}+{\bf a})}.
\end{equation} 
Then the right hand side of (\ref{eq:FE_T_n_rewrite}) can be defined on 
\begin{equation}\label{convergence_range_T_D_1}
\mathcal{D}_1=\left\{(s,t)\in \C^2 \,:\, 
\max_{j}\left\{\frac{-1-c_j}{a_j}\right\}-1<\textup{Re}(t)<\textup{Re}(s)+\min_{j}\left\{\frac{b_j+\frac{m}{n}-1-c_j}{a_j}\right\}+1\right\},
\end{equation}
which contains $\mathcal{D}_0$. Therefore, the left hand side of (\ref{eq:FE_T_n_rewrite}) can be extended to $\mathcal{D}_1$
as a meromorphic function in $(s, t)$.
We apply the equation  (\ref{eq:FE_T_n_rewrite}) repeatedly to extend the defining range of $\T_n$ to all of $(s, t)\in \C^2$ meromorphically.
\end{proof}

\begin{proof}[Proof of Theorem \ref{thm:mero_non}]
We can regard $\bar{\Z}(F_s,t)$ as a finite linear combination of integrals of the form (\ref{eq:z_polar_sign}), which is
\begin{equation}\label{eq:z_polar_sign_T}
\int_{K\cap L}\T_n\left(s{\bf 1},t{\bf 1}+{\bf c},{\bf x},(k^{-1}\cdot\varphi)[{\bf x}]\right)dk.
\end{equation}
Here, $K\cap L$ is compact.
Therefore, we may apply an analogous argument with the analytic part of the proof of Proposition \ref{prop:mero_cont_noncompact}.
Furthermore, the meromorphic part follows from the fact that 
the integrand of (\ref{eq:z_polar_sign_T}) has a meromorphic continuation to all of $(s,t)$ in $\C^2$.
\end{proof}

\subsection{Proof of Theorem \ref{thm:extended_F_E}}
We begin with the definition of $\nu(M,N)(F)$ on $C^N(\bar{\fn})$: 
\begin{equation*}
\nu(M,N)(F)=\sup_{X\in \bar{\fn}}\left\{ (1+\|X\|^2)^M \cdot \sum_{\alpha,|\alpha|\leqslant N} |\partial^{\alpha}F(X)| \right\}\,.
\end{equation*}
It is known that the functional equation (\ref{eq:FT_S}) still holds for functions satisfying certain decay condition.
Precisely, we have the following proposition.
\begin{Prop}\textup{\cite[Lemma 4.34]{Kimura03}}
Suppose $F\in C^{\infty}(\bar{\fn})$ satisfies $\nu(M_0+1,M_0)(F)<\infty$ for sufficiently large $M_0$.
Then the functional equation (\ref{eq:FT_S}) holds as meromorphic functions in $t$:
\begin{equation*}
\frac{\pi^{\frac{nt}{2}}}{\Gamma_n(t)}\bar{\Z}(F,t-\frac{m}{n})
=\frac{\pi^{\frac{n}{2}(-t+\frac{m}{n})}}{\Gamma_n(-t+\frac{m}{n})}\Z(\widehat{F},-t).
\end{equation*}
\end{Prop}
Note also that functions in $I(s)$ have a decay condition for sufficiently large $\textup{Re}(s)$ by Lemma \ref{lem_F_s_phi}.

\begin{Lem}\label{lem:nu_fourier}
Let $\textup{Re}(s)\geqslant 2M-\frac{m}{n}$.
If $F_s\in I(s)$, then $\nu(M,N)(F_s)<\infty$ for all $N\geqslant 0$.
\end{Lem}
\begin{proof}
For $X\in \bar{\fn}$, we write $X=m\cdot\sum_{j=1}^nx_jF_j$, $m\in M$.
Then we have
\begin{align*}
\| X\|^2&=\left\langle \sum_{j=1}^nx_jF_j, \sum_{j=1}^nx_jF_j \right\rangle\\
&=\frac{n}{4m}\Tr\left(\ad\left(\sum_{j=1}^nx_jF_j\right)\ad\left(\sum_{j=1}^nx_jE_j\right)\right)\\
&=\frac{n}{4m}\cdot 4(d(n-1)+(e+1))\sum_{j=1}^nx_j^2\\
&=\sum_{j=1}^nx_j^2.
\end{align*}

On the other hand, since each $\partial^{\alpha}$, $|\alpha|=1$, acts as an element in the enveloping algebra $\mathcal{U}(\fg)$, so $\partial^{\alpha}(F_s)\in I(s)$ for any multi-index $\alpha$. 
Then $\partial^{\alpha}(F_s)(X)$ is of the form $\frac{\varphi(\kappa(\sum_{j=1}^nx_jF_j))}{\prod_{j=1}^n(1+x_j^2)^{\frac{s+\frac{m}{n}}{2}}}$ for some $\varphi\in C^{\infty}(K/M)$. 

Therefore, for any nonnegative integer $N$, we have
\begin{align*}
\nu(M,N)(F_s)
&=\sup_X\left\{\left(1+\|X\|^2 \right)^{M}\cdot\sum_{\alpha,|\alpha|\leqslant N}\left|\partial^{\alpha}F_s\right| \right\}\\
&\leqslant c\sup_{(\R^+)^n}\left\{\frac{\left(1+\sum_{j=1}^nx_j^2\right)^M}{\prod_{j=1}^n(1+x_j^2)^{\frac{\Re(s)+\frac{m}{n}}{2}}}\right\},
\end{align*}
which is finite because $M\leqslant \frac{\Re(s)+\frac{m}{n}}{2}$. 
\end{proof}

\begin{Cor}\label{cor:Kimura}
For sufficiently large $\Re(s)$, the functional equation
\begin{equation}\label{eqn:F_E_I(s)}
\frac{\pi^{\frac{nt}{2}}}{\Gamma_n(t)}\bar{\Z}(F_s,t-\frac{m}{n})
=\frac{\pi^{\frac{n}{2}(-t+\frac{m}{n})}}{\Gamma_n(-t+\frac{m}{n})}\Z(\widehat{F_s},-t).
\end{equation}
holds for functions $F_s$ in $I(s)$ as meromorphic functions in $t\in \C$.
\end{Cor}
We combine Theorem \ref{thm:mero_non} and Corollary \ref{cor:Kimura} to conclude Theorem \ref{thm:extended_F_E}.

\section{Applications to the representation theory}
The standard intertwining map $\widetilde{B}_s: \Ind_P^G(s)\longrightarrow \Ind_P^G(-s)$ is defined by
\begin{equation*}
\widetilde{B}_s(f)(X)=\int_{\bar{\fn}} f(\bar{n}_Xw \bar{n}_{X_1})dX_1
\end{equation*} 
on which the integral converges. See, for example, \cite[pg 174]{Knapp86}.
We also define $\widetilde{A}_s$ on $I(s)$ by
\begin{align*}
\widetilde{A}_s(F_s)(X)&=\bar{\Z}\left(\tau_XF_s, s-\frac{m}{n}\right)\\
&=\int_{\bar{\fn}} F_s(X_1)\bar{\nabla}(X_1-X)^{s-\frac{m}{n}} dX_1\,,
\end{align*}
where $\tau_XF_s=F_s(\cdot+X)$.
Then the integral $\widetilde{A}_s(F_s)(X)$ 
converges absolutely for $\Re(s)>d(n-1)$ 
by Lemma \ref{lem:z_conv} and it can be derived from the integral $\widetilde{B}_s(f)(X)$. 

\begin{Lem}\textup{(\cite[pg 183]{Knapp86})}\label{lem:As_Bs}
Let $F_s\in I(s)$ be the function corresponding to $f\in \textup{Ind}_P^G(s)$. Then 
\begin{equation*}
\widetilde{A}_s(F_s)=\widetilde{B}_s(f)\quad \textup{for  } \Re(s)>d(n-1).
\end{equation*}
\end{Lem}
\noindent Therefore, $\widetilde{A}_s$ is a $G$-intertwining operator from $I(s)$ to $I(-s)$ for $\Re(s)>d(n-1)$.  
Moreover, Lemma \ref{lem:intertwining} below is a well-known fact 
(see, for example, \cite{Knapp86}, \cite{VoganWallach90}) and we give an alternative proof.

\begin{Lem}\label{lem:intertwining}    
The intertwining operators $\widetilde{A}_s$ are complex analytic in $s$ for $\Re(s)>d(n-1)$ and 
have meromorphic continuations to all of $\C$
\end{Lem}
\begin{proof}
The meromorphic continuation part follows by setting $t=s-\frac{m}{n}$ in Theorem \ref{thm:mero_non}.
\end{proof}

We now consider the case $G=GL(2n,\R)$ and a standard hermitian form on $I(s)$.
We define $\langle \,\,,\, \rangle_{s,t}$ on $I(s)$ by
\begin{equation}\label{def:form_st}
\langle F_s,G_s\rangle_{s,t}
=\frac{\pi^{\frac{n(t+4n)}{2}}}{\Gamma_n(t+4n)}\int_{\bar{\fn}}F_s(X)
\overline{\bar{\Z}(\tau_XG_{\bar{s}},\bar{t}-n)}dX\,,
\end{equation} 
which is a $G$-invariant hermitian form if $s=t\in \R$. See \cite[Proposition 14.23]{Knapp86}, for example.
For $\Re(s)>n-1$, we also define a $L^2$-norm on $V_s=\left\{\widehat{F_s}: F_s\in I(s)\right\}$ by
\begin{equation*}
\langle \widehat{F_s},\widehat{G_s} \rangle_{L^2\left(\nabla^{-t}\right)}
=\int_{\fn}\widehat{F_s}(Y)\overline{\widehat{G_{\bar{s}}}(Y)}\nabla(Y)^{-t}dY\,.
\end{equation*}
As an application of Theorem \ref{thm:extended_F_E}, we prove the following proposition.

\begin{Prop}\label{prop:herm_mero}
As meromorphic functions in $s$ and $t$ on $\Omega_1\cup\Omega_2$,
\begin{equation*}
\langle F_s,G_s\rangle_{s,t}=
\frac{\pi^{\frac{n(-t-3n)}{2}+4n}16^n}{\Gamma_n(-t-3n)\cdot b_{2n}(t+3n)}
\langle \widehat{F_s},\widehat{G_s} \rangle_{L^2\left(\nabla^{-t}\right)}\,,
\end{equation*}
where
\begin{align*}
\Omega_1&=\bigcup_{k\in \Z_{\geq 0}}\left\{(s,t)\in \C^2: n-2k<\Re(t)<\Re(s)+1-2k\right\}
\textup{ and }\\
\Omega_2&=\left\{(s,t)\in \C^2: \Re(s)>n-1\textup{ and }\Re(t)<1\right\}\,.
\end{align*}
\end{Prop}

We need some lemmas to prove Proposition \ref{prop:herm_mero}.
\begin{Lem}\label{lem:form_st}
$\langle \,\,,\, \rangle_{s,t}$ is complex analytic on $\{(s,t)\in \C^2:n<\Re(t)<\Re(s)+1\}$ 
and it can be extended to a meromorphic function on $\Omega_1$
\end{Lem}
\begin{proof}
We assume that $\Re(t)>n$.
We write $X=k\cdot \sum_{j=1}^nx_jF_j$ with $x_j>0$ 
and $Y=k'\cdot \sum_{j=1}^ny_jF_j$ with $y_j>0$ for some $k, k'\in K\cap L$.
Since each entry in $k$ and $k'$ is a number between $-1$ and $1$, we have
\begin{align*}
|\det(X-Y)|^{\Re(t)-n}
\leq \left[\sum_{\sigma\in S_n}\prod_{j=1}^n\left(x_j+y_{\sigma(j)}\right)\right]^{\Re(t)-n}
\leq c\left[\sum_{\sigma\in S_n}\prod_{j=1}^n
\left(x_j^{\Re(t)-n}+y_{\sigma(j)}^{\Re(t)-n}\right)\right]\,.
\end{align*}
By Proposition \ref{prop_polar_coord} and (\ref{eq:z_polar_sign}), we have
\begin{equation*}
|\langle F_s,G_s\rangle_{s,t}|\leq c'\sum_{\sigma\in S_n}\int_{(\R^+)^n}\int_{(\R^+)^n}\prod_{j=1}^n
\frac{\left(x_j^{\Re(t)-n}+y_{\sigma(j)}^{\Re(t)-n}\right)(x_jy_j)^{2n-2j}}
{\left\{(1+x_j^2)(1+y_j^2)\right\}^{\frac{\Re(s)+n}{2}}}
\prod_{j=1}^ndx_j\prod_{j=1}^ndy_j
\end{equation*}
Each of these integrals is bounded by a product of one variable integrals of the form
\begin{equation*}
\int_0^{\infty}\frac{x^{2n-2j}}{(1+x^2)^{\frac{\textup{\Re}(s)+n}{2}}}dx\textup{ and }
\int_0^{\infty}\frac{x^{\Re(t)+n-2j}}{(1+x^2)^{\frac{\textup{\Re}(s)+n}{2}}}dx
\textup{ for all }j=1,\cdots,n.
\end{equation*}
They are finite if $\Re(t)<\Re(s)+1$.
The analytic part follows from the standard application of Morera's Theorem 
as in the proof of Proposition \ref{prop:mero_cont_noncompact}.

By (\ref{eq:b_k}), as convergent integrals
\begin{equation}\label{eq:b_k_shifted}
\bar{\Z}\left(\tau_XG_{\bar{s}}\,,\bar{t}-n\right)=\frac{1}{b_k\left(\bar{t}-n+2k\right)}
\bar{\Z}\left(\bar{\nabla}(\partial_X)^{2k}\tau_XG_{\bar{s}}\,,\bar{t}-n+2k\right)
\end{equation}
for $n<\Re(t)<\Re(s)+1-2k$.
This proves the meromorphic part. 
\end{proof}

We recall the facts that $F_s\in L^1$ for $F_s\in I(s)$ with $\Re(s)>n-1$ and
$\nabla^{-t}$ is a locally $L^1$-function for $\Re(t)<1$.

\begin{Lem}\label{lem:L^2norm_dmu}
For $F_s$, $G_s\in I(s)$, the integral 
$\langle \widehat{F_s},\widehat{G_s} \rangle_{L^2\left(\nabla^{-t}\right)}$
is complex analytic on $\Omega_2$ and we have 
\begin{equation*}
\langle \widehat{F_s},\widehat{G_s} \rangle_{L^2\left(\nabla^{-t}\right)}
=\int_{\bar{\fn}}F_s(X)\overline{\Z(\widehat{\tau_XG_{\bar{s}}},-\bar{t})}dX\textup{ on }\Omega_2\,.
\end{equation*}
\end{Lem}
\begin{proof}
For $F_s\in I(s)$ with $\Re(s)>n-1$, we show that there exist constants $c_k>0$ so that
\begin{equation}\label{eq:Fourier_F_s_bdd}
\left|\widehat{F_s}(Y)\right|\leq \frac{c_k}{\left(1+\|Y\|\right)^k}\textup{ for all }Y\in \fn.
\end{equation}
For any multi-index $\alpha$, we have
\begin{equation}\label{eq:FT_Diff}
\widehat{\partial^{\alpha}(F_s)}\in I(s)\textup{ and }
\widehat{\partial^{\alpha}(F_s)}(Y)=(2\pi i Y)^{\alpha}\widehat{F_s}(Y).
\end{equation}
Moreover, $\widehat{\partial^{\alpha}(F_s)}\in C_0$ by the Riemann-Lebesgue Lemma
and this proves (\ref{eq:Fourier_F_s_bdd}).
Therefore, 
the integral $\langle \widehat{F_s},\widehat{G_s} \rangle_{L^2\left(\nabla^{-t}\right)}$
converges absolutely for $\Re(t)<1$. 
The analytic part is from Morera's Theorem 
as in the proof of Proposition \ref{prop:mero_cont_noncompact}
together with the fact that the integrand of 
$\langle \widehat{F_s},\widehat{G_s} \rangle_{L^2\left(\nabla^{-t}\right)}$
is complex analytic in $s$ and $t$ on $\Omega_2$.

By Fubini's Theorem, we have the following string of equalities on $\Omega_2$ as convergent integrals.
\begin{align*}
\langle \widehat{F_s},\widehat{G_s} \rangle_{L^2\left(\nabla^{-t}\right)}
&=\int_{\bar{\fn}}\int_{\fn}F_s(X)e^{-2\pi i\langle X,Y\rangle}
\overline{\widehat{G_{\bar{s}}}(Y)\nabla(Y)^{-\bar{t}}}dYdX\\
&=\int_{\bar{\fn}}F_s(X)\overline{\Z\left(\widehat{\tau_XG_{\bar{s}}}\,,-\bar{t}\right)}dX\,.
\end{align*}
This proves the lemma.
\end{proof}

\noindent By (\ref{eq:FT_Diff}), for $F_s\in I(s)$ with $\Re(s)>n-1$, we also have
\begin{equation}\label{eq:Diff_Poly}
\left(\bar{\nabla}(\partial_X)^{2k}F_s\right)^{\widehat{}}(Y)
=\left(2\pi i \nabla(Y)\right)^{2k}\,\widehat{F_s}(Y).
\end{equation}

\noindent We can now finish the proof of Proposition \ref{prop:herm_mero}.

\begin{proof}[Proof of Proposition \ref{prop:herm_mero}]
We set $k=2n$ in the equation (\ref{eq:b_k_shifted}), by Lemma \ref{lem:form_st}, we get
\begin{equation}\label{eq:b_k_shifted_Z}
\langle F_s,G_s\rangle_{s,t}
=\frac{\pi^{\frac{n(t+4n)}{2}}}{\Gamma_n(t+4n)\cdot b_n(t+3n)}\int_{\bar{\fn}}F_s(X)
\overline{\bar{\Z}\left(\bar{\nabla}(\partial_X)^{4n}\tau_XG_{\bar{s}}\,,\bar{t}+3n\right)}dX
\end{equation}  
on the range $n<\Re(t)<\Re(s)+1-4n$ as convergent integrals.
On the other hand, the integral on the right hand side of (\ref{eq:b_k_shifted_Z}) 
converges absolutely on $\left\{(s,t)\in \C^2: -3n<\Re(t)<\Re(s)+1-4n\right\}$, 
which intersects with $\Omega_2$. 
By Theorem \ref{thm:extended_F_E} and  (\ref{eq:Diff_Poly}),
as meromorphic functions on the intersection
\begin{align*}
\frac{\pi^{\frac{n(t+4n)}{2}}}{\Gamma_n(t+4n)}&\int_{\bar{\fn}}F_s(X)
\overline{\bar{\Z}\left(\bar{\nabla}(\partial_X)^{4n}\tau_XG_{\bar{s}}\,,\bar{t}+3n\right)}dX\\
&=\frac{\pi^{\frac{n(-t-3n)}{2}}}{\Gamma_n(-t-3n)}\int_{\bar{\fn}}F_s(X)
\overline{\Z\left(\left(\nabla(\partial_X)^{4n}\tau_XG_{\bar{s}}\right)^{\widehat{}},
-\bar{t}-4n\right)}dX\\
&=\frac{\pi^{\frac{n(-t-3n)}{2}+4n}16^n}{\Gamma_n(-t-3n)}\int_{\bar{\fn}}F_s(X)
\overline{\Z\left(\left(\tau_XG_{\bar{s}}\right)^{\widehat{}},-\bar{t}\right)}dX\,.
\end{align*}
Then the proposition follows from Lemma \ref{lem:L^2norm_dmu}.
\end{proof}

\appendix
\section{Tables}\label{table}
The following two tables give information on the groups under consideration in this paper.

\begin{table}[h]
\begin{center}
\begin{tabular}{|l|c|c|c|c|c|}
\hline
\null&$G$&$n=\textup{rank}(\fn)$&$m=\dim(\fn)$&$d$&$e$\\ \hline
$1$.&$GL(2n,\R), n\geq 2$&$n$&$n^2$&$1$&$0$\\
$2$.&$O(2n,2n)$&$n$&$n(2n-1)$&$2$&$0$\\
$3$.&$E_7(7)$&$3$&$27$&$4$&$0$\\
$4$.&$O(p,q),\,p,q\geq 3$&$2$&$p+q-2$&$(p+q-4)/2$&$0$\\ \hline
$5$.&$Sp(n,\C)$&$n$&$n(n+1)$&$1$&$1$\\
$6$.&$SL(2n,\C)$&$n$&$2n^2$&$2$&$1$\\
$7$.&$SO(4n,\C)$&$n$&$2n(2n-1)$&$4$&$1$\\
$8$.&$E_{7,\C}$&$3$&$54$&$8$&$1$\\
$9$.&$SO(p,\C),\,p\geq 3$&$2$&$2(p-2)$&$p-4$&$1$\\ \hline
$10$.&$Sp(n,n)$&$n$&$n(2n+1)$&$2$&$2$\\
$11$.&$GL(2n,{\bf H})$&$n$&$4n^2$&$4$&$3$\\
$12$.&$SO(p,1)$&$1$&$p$&$0$&$p-1$\\
\hline
\end{tabular}
\end{center}
\vskip.2in
\caption{\null}
\end{table}

\begin{table}[h]
\begin{center}
\begin{tabular}{|l|c|c|c|}
\hline
\null&$V\simeq \fn$&$L$&$\nabla$\\ \hline
$1$.&$M(n\times n,\R)$&$GL(n,\R)\times GL(n,\R)$&$|\det|$\\
$2$.&$Skew(2n:\R)$&$GL(2n,\R)$&$\textup{Pfaffian}$\\
$3$.&$Herm(3,{\bf O}_{split})$&$E_6(6)\times \R^{\times}$&$\textup{deg. $3$ poly}$\\
$4$.&$\R^{p-1,q-1}$&$\R^{\times}O(p-1,q-1)$&$(X,X)$\\ \hline
$5$.&$Sym(n,\C)$&$GL(n:\C)$&$|\det|$\\
$6$.&$M(n\times n,\C)$&$S(GL(n,\C)\times GL(n,\C))$&$|\det|$\\
$7$.&$Skew(2n,\C)$&$GL(2n,\C)$&$|\textup{Pfaffian}|$\\
$8$.&$Herm(3,{\bf O})_{\C}$&$E_{6,\C} \C^{\times}$&$|\textup{deg. $3$ poly}|$\\
$9$.&$\C^{p-1}$&$SO(p-2:\C)\times \C^{\times}$&$|(Z,Z)|$\\ \hline
$10$.&$Sym(2n,\C)\cap M(n\times n,{\bf H})$&$GL(n,{\bf H})$&$|\det_{\C}(Z)|^{\frac{1}{2}}$\\
$11$.&$M(n\times n,{\bf H})$&$GL(n,{\bf H})\times GL(n,{\bf H})$&$|\det_{\C}(Z)|^{\frac{1}{2}}$\\
$12$.&$\R^{p-1}$&$SO(p-1)\times \R^{\times}$&$\|\cdot\|$\\
\hline
\end{tabular}
\end{center}
\vskip.2in
\caption{Jordan algebras for the groups in Table $1$.}
\end{table}

\end{document}